\documentclass[reqno]{amsart}
\usepackage{amsmath}
\usepackage{hyperref}
\usepackage{amsthm}
\usepackage{amssymb}
\usepackage{graphics}
\usepackage{latexsym}
\usepackage{comment}
\usepackage{extarrows}
\usepackage{wrapfig}
\usepackage{comment}
\usepackage {colortbl,array,xcolor}
\usepackage{hhline}
\usepackage{tikz}
\usetikzlibrary{intersections, calc, math, patterns}

\numberwithin{equation}{section}
\newtheorem{thm}{Theorem}[section]
\newtheorem{prop}[thm]{Proposition}
\newtheorem{lem}[thm]{Lemma}

\theoremstyle{remark}

\newcommand{\R}{{\mathbb R}}

\newcommand{\N}{{\mathbb N}}

\newcommand{\Sp}{{\mathbb S}}
\newcommand{\LR}[1]{{\langle {#1} \rangle }}

\newcommand{\F}{\mathcal{F}}

\newcommand{\supp}{\operatorname{supp}}

\newcommand{\vece}{{\textnormal{\textbf{e}}}}

\title[LWP of a system describing laser-plasma intercations]{Local well-posedness of a system describing laser-plasma interactions}

\author[S.~Herr]{Sebastian Herr}
\address[S.~Herr]{Universit\"{a}t Bielefeld, Fakult\"{a}t f\"{u}r Mathematik, Postfach 10 01 31, 33501 Bielefeld, Germany}
\email{herr@math.uni-bielefeld.de}

\author[I.~Kato]{Isao Kato}
\address[I.~Kato]{Department of Mathematics, Graduate School of Science, Kyoto University,
Kyoto 606-8502, Japan}
\email{kato.isao.23n@st.kyoto-u.ac.jp}
\author[S.~Kinoshita]{Shinya Kinoshita}
\address[S.~Kinoshita]{Department of Mathematics, Graduate School of Science and Engineering, Saitama University, Saitama 338-8570, Japan}
\email{kinoshita@mail.saitama-u.ac.jp}
\author[M.~Spitz]{Martin Spitz}
\address[M.~Spitz]{Universit\"{a}t Bielefeld, Fakult\"{a}t f\"{u}r Mathematik, Postfach 10 01 31, 33501 Bielefeld, Germany}
\email{mspitz@math.uni-bielefeld.de}


\begin{document}

\begin{abstract}
A degenerate Zakharov system arises as a model for the description of laser-plasma interactions. It is a coupled system of a Schr\"{o}dinger and a wave equation with a non-dispersive direction. In this paper, a new local well-posedness result for rough initial data is established. The proof is based on an efficient use of local smoothing and maximal function norms.
\end{abstract}

\maketitle


\section{Introduction}\label{sec:intro}

In view of numerous applications, there is strong interest in plasma dynamics and laser-plasma interactions. Ideally, one wants to use numerical simulations to gain insight in these processes. This requires reliable models and a thorough understanding thereof.

In 1972, Zakharov introduced in \cite{Z72} the system
\begin{equation}
\label{eq:Zakharov}
	\begin{split}
	i\partial_t E + \Delta E ={}& E n \quad \text{ in }(-T,T) \times \R^d, \\
	\partial^2_t n - \Delta n ={}& \Delta |E|^2 \quad \text{ in }(-T,T) \times \R^d,
	\end{split}
\end{equation}
to study Langmuir waves in a non- or weakly magnetized plasma, where the physical dimension is $d=3$. Here, $E$ denotes the complex envelope of the electric field and $n$ the ion density fluctuation.

A different situation arises when modelling the interaction of a plasma with a laser beam.
Using the paraxial approximation (see e.g.~\cite[Section~4]{S05}) to describe this interaction, one obtains
the system
\begin{equation}\label{eq:dZakharov}
  \begin{split}
  i(\partial_t E + \partial_{x_d} E) + \Delta' E={}&nE
  \quad \text{ in }(-T,T) \times \R^d,\\
\partial_t^2 n - \Delta' n ={}& \Delta' |E|^2  \quad \text{ in }(-T,T) \times \R^d,
\end{split}
\end{equation}
where $E$ now denotes the complex amplitude of the laser beam and $n$ the real-valued electron density fluctuation. Both are functions of the variables $(t,x_1,\ldots,x_d) \in(-T,T) \times \R^d$. Since the last spatial variable $x_d$ plays a distinguished role (the direction of propagation of the laser beam), we use the notation $x=(x_1,\ldots,x_{d-1}) \in \R^{d-1}$ and $\Delta'=\sum_{i=1}^{d-1}\partial_{x_i}^2$.
 We refer to~\cite{RD93} and~\cite{SS99} for a derivation in $d=3$. In~\cite{RD93} a reduced version of~\eqref{eq:dZakharov} was used to analyze self-focusing from local intensity peaks (hot spots) in laser plasmas, which is a possible instability for inertial confinement fusion.

A more precise description of laser-plasma interaction takes into account that part of the incident light field is backscattered by Raman- and Brillouin-type processes. The three resulting light fields interact with the electric field of the plasma as well as with the density fluctuation. The resulting system can be seen as a nonlinear coupling of equations of the form~\eqref{eq:Zakharov} and~\eqref{eq:dZakharov}. A reduced model system of this type was used in~\cite{RDR99} for numerical simulations, see also~\cite{CC04}. The first step in the analysis of these advanced models is the understanding of systems~\eqref{eq:Zakharov} and~\eqref{eq:dZakharov}.
Finally, we note that the system~\eqref{eq:dZakharov} also arises as WKB approximation for the Euler-Maxwell equations in the cold ion case for highly oscillatory initial data, see~\cite{T05}.

In the present paper, we study the initial value problem associated with \eqref{eq:dZakharov}, i.e. we prescribe
\begin{equation}\label{eq:ic}
  (E,n, \partial_t n)|_{t=0} =(E_0,n_0,n_1) \quad \text{ in } \R^d.
  \end{equation}

We prove the following local well-posedness result.
\begin{thm}\label{thm:main}
  Let $d \geq 3$, $s>\frac{d-2}{2}$, $s'>\frac12$. Then, \eqref{eq:dZakharov}--\eqref{eq:ic} is locally well-posed if the initial data satisfies
$$ (E_0,n_0,|\nabla'|^{-1}n_1)\in H^{s,s'}(\R^d) \times \big(H^{s-\frac12, s'} (\R^d)\big)^2.$$
\end{thm}
We define the non-isotropic Sobolev spaces $H^{s,s'}(\R^d)$ as the collection of all $f \in \mathcal{S}'(\R^d)$ satisfying
\[
\| f \|_{H^{s, s'}(\R^d) }  := \Bigl( \int_{\R^d} \LR{\xi}^{2 s} \LR{\xi_d}^{2 s'} | \F_{x,x_d} f (\xi, \xi_d) |^2 d\xi d \xi_d
 \Bigr)^{1/2}<+\infty ,
\]
where $\xi = (\xi_1,\ldots,\xi_{d-1}) \in \R^{d-1}$, $\xi_d \in \R$, and $|\nabla'|^{-1}=(\sqrt{-\Delta'})^{-1}$ is the Fourier multiplier. We refer to Theorem \ref{thm:lwp} for a more precise version of our main result.

Without going into detail, we remark that our proof in Section \ref{sec:NonlinearEstimates} also implies certain refinements in Besov spaces at the threshold regularities if $d\geq 4$, and in addition, one could avoid low frequency conditions (see Section \ref{sec:lwp}).


\subsection*{Previous results}
Coupling two of the fundamental dispersive equations, the Zakharov system~\eqref{eq:Zakharov} and the corresponding initial value problem have attracted a lot of attention. We refer to~\cite{GGKZ16, SS99} and the references therein for the history of the problem and to~\cite{OT92, BC96, GTV97} for a few milestones in the theory. The local well-posedness theory for the Zakharov system is now comprehensively understood, see~\cite{CHN19} for the state of the art in dimensions $d \geq 4$ and~\cite{S21} for $d \leq 3$.

Due to the lack of dispersion in the longitudinal direction in~\eqref{eq:dZakharov} the system \eqref{eq:dZakharov} is sometimes called the \emph{degenerate Zakharov system}. This partial lack of dispersion adds significant difficulties to the well-posedness theory, which therefore is still in its infancy.  In~\cite{CC04} the question of local well-posedness of~\eqref{eq:dZakharov} has been posed. The periodic problem for~\eqref{eq:dZakharov} is ill-posed, see~\cite{CM06}. A positive answer in dimension three was given in~\cite{LPS05} for initial values $(E_0, n_0,n_1)$ in $H^5(\R^3)\times H^5(\R^3)\times H^4(\R^3)$ with $\partial_{x_1}^{\frac{1}{2}}E_0, \partial_{x_2}^{\frac{1}{2}}E_0 \in H^5(\R^3)$ and $\partial_{x_3} n_1 \in H^4(\R^3)$, using local smoothing and maximal function estimates. Improving upon the maximal function estimate, local well-posedness for initial values $(E_0, n_0, n_1) \in H^2(\R^3) \times H^2(\R^3) \times H^1(\R^3)$ with $\partial_{x_1}^{\frac{1}{2}}E_0, \partial_{x_2}^{\frac{1}{2}}E_0 \in \dot{H}^2(\R^3)$ and $\partial_{x_3} n_1 \in H^1(\R^3)$ was shown in~\cite{BL15}.

In view of these results, the assumptions on the initial data in Theorem \ref{thm:main} are lowered significantly. Our approach is based on an efficient use of local smoothing and maximal function norms. More precisely, we adapt the approach devised in \cite{BIKT11} (to solve the Schr\"{o}dinger maps problem) to the setting of the degenerate Zakharov system.

\subsection*{Organisation of the paper}
In Section \ref{sec:pre} we introduce notation and provide linear estimates. In Section \ref{sec:lwp} we prove the main result under the hypothesis that two nonlinear estimates hold, which we then prove in Section
\ref{sec:NonlinearEstimates}. In an appendix, we complement our results by showing that it is impossible to prove the nonlinear estimates in Fourier restriction norms only.

\section{Preliminaries}\label{sec:pre}
\subsection*{Notation}
Throughout the paper, we use the following notations.
$A{\ \lesssim \ } B$ means that there exists $C>0$ such that $A \le CB.$
Also, $A\sim B$ means $A{\ \lesssim \ } B$ and $B{\ \lesssim \ } A.$
Let $u=u(t,x,x_d)$ and let $\ \F_t u,\ \F_x, \ \F_{x,x_d} u$ denote the Fourier transform of $u$ in time, $\R^{d-1}$, and $\R^d$, respectively. By
$\F_{t, x,x_d} u = \widehat{u}$ we denote the Fourier transform of $u$ in time and space.
Let $N$, $M$ be dyadic numbers, i.e. there exist $n_1, m_1 \in \N_{0}$ such that
$N= 2^{n_1}$ and $M=2^{m_1}$. Let $\eta \in C^{\infty}_{0}((-2,2))$ be an even, non-negative function which satisfies $\eta (t)=1$ for $|t|\leq 1$. Letting
$\eta_N (\xi):=\eta (|\xi| N^{-1})-\eta (|\xi|2 N^{-1})$,
$\eta_1(\xi):=\eta (|\xi|)$,
the equality ${\sum_{N}\eta_{N}=1}$ holds.
Here we used ${\sum_{N}= \sum_{N \in 2^{\N_0}}}$ for simplicity.
We also use the abbreviations ${\sum_{M}= \sum_{M \in 2^{\N_0}}}$,
 ${\sum_{N, M}= \sum_{N, M \in 2^{\N_0}}}$, etc. throughout the paper.

Let $\textbf{e} \in \Sp^{d-2}$ and $\mathcal{P}_{\vece}=\{\xi \in \R^{d-1} \, |\, \xi \cdot \textbf{e}=0 \}$ with the induced Euclidean measure.
For $p$, $q \in [1,\infty]$, define
\[
\| f \|_{L_{\textbf{e}}^{p,q}} = \Bigl( \int_{\R} \Bigl( \int_{\R \times \mathcal{P}_{\textbf{e}}} |f(t, r \textbf{e} + v)|^q dt dv \Bigr)^{p/q} dr \Bigr)^{1/p}.
\]
We define $I_N^{(d-1)} = \{ \xi \in \R^{d-1} \, | \, \xi \in \supp \eta_N\}$. Let $T >0$ and
\[
L_N^2(T)=\{f \in L^2([-T,T]\times \R^{d-1}) \, | \, \supp \F_{x}f \subset [-T,T] \times I_{N}^{(d-1)}\}.
\]
Let $\phi \in C_0^\infty(\R)$ be non-negative and symmetric, such that $\phi(r)=0$ if $|r|\leq (4\sqrt{d-1})^{-1}$ or $|r|>4$ and $\phi(r)=1$ if $ (2\sqrt{d-1})^{-1}\leq r\leq 2$, and we set $\phi_N(r)=\phi(r/N)$. Then,
\begin{equation} \label{eq:DecompositionIdentity}
\prod_{j=1}^{d-1}(1-\phi_N(\xi_j))=0 \text{ for all } \xi=(\xi_1,\ldots,\xi_{d-1}) \in I_N^{(d-1)} \text{ and } N \in 2^{\N}.
\end{equation}
 We define $P_N = \F_{x}^{-1} \eta_N \F_{x}$ and $P_{N,\vece} = \F_{x}^{-1} \phi_N(\xi \cdot \vece) \F_{x}$. Since both $P_N$ and $P_N P_{N,\vece}$ have kernels in $L^1 (\R^{d-1})$, they are bounded operators on each of the spaces $L_{\textbf{e}'}^{p,q}$.

Let $d \geq 3$, $T>0$, and $p_{d} = (2 d+4)/d$.  For $N>1$ and $f \in L_N^2(T)$, we define the norms
\begin{align*}
\|f\|_{F_N(T)} = \|f\|_{L_t^{\infty}L_x^2} + \|f \|_{L_{t,x}^{p_{d-1}}} + N^{-\frac{d-2}{2}} \sum_{j=1}^{d-1} \|f\|_{L_{\vece_j}^{2,\infty}} + N^{\frac12} \sum_{j=1}^{d-1} \|P_{N,\vece_j}f\|_{L_{\vece_j}^{\infty,2}}
\end{align*}
if $d \geq 4$ and
\begin{align*}
\|f\|_{F_N(T)} &= \|f\|_{L_t^{\infty}L_x^2} + \|f \|_{L_{t,x}^{4}} + (\log N)^{-1}N^{-\frac12} \sum_{j=1}^{2} \|f\|_{L_{\vece_j}^{2,\infty}} \\
&\qquad + N^{\frac12} \sum_{j=1}^{2} \|P_{N,\vece_j}f\|_{L_{\vece_j}^{\infty,2}}
\end{align*}
in the case $d = 3$. To estimate the nonlinear terms, we introduce
\begin{align*}
 \|g\|_{G_N(T)} = \inf_{g=g_1+g_2} \Big(\|g_1 \|_{L_{t,x}^{p_{d-1}'}} + N^{-\frac12} \sum_{j = 1}^{d-1} \|g_2\|_{L_{\vece_j}^{1,2}}\Big).
\end{align*}
Here $p_{d-1}'$ satisfies $1/p_{d-1} + 1/p_{d-1}' = 1$ and $\vece_1, \ldots, \vece_{d-1}$ denote the standard basis of $\R^{d-1}$.
For $N=1$ we modify the above definition as follows:
\[\|f\|_{F_1(T)} = \|f\|_{L_t^{\infty}L_x^2} + \|f \|_{L_{t,x}^{p_{d-1}}}+\sum_{j=1}^{d-1} \|f\|_{L_{\vece_j}^{2,\infty}}, \qquad \|g\|_{G_1(T)} = \|g \|_{L_{t,x}^{p_{d-1}'}}. \]
For $T>0$ and $s$, $s' \geq 0$, we define the normed spaces $F^{s,s'}(T)$ and $W^{s,s'}(T)$ as
\begin{align*}
 F^{s,s'}(T) &= \{ f \in L^2([-T,T] \times \R^{d}) \,| \, \\
&\qquad \|f\|_{F^{s,s'}(T)} = \bigl( \sum_{N,M} M^{2 s'}N^{2s} \bigl\| \eta_{M}(\xi_d)  \| P_N  \F_{x_d}f\|_{F_N(T)}\bigr\|_{L_{\xi_d}^2}^2\bigr)^{\frac12} < \infty\},\\
 W^{s,s'}(T) &= \{ f \in L^2([-T,T] \times \R^{d}) \,| \\
 &\qquad \|f\|_{W^{s,s'}(T)} =  \bigl( \sum_{N,M} M^{2 s'}N^{2s}  \bigl\| \eta_{M}(\xi_d)  \| P_N  \F_{x_d}f\|_{L_t^{\infty}L_x^2}\bigr\|_{L_{\xi_d}^2}^2\bigr)^{\frac12}   < \infty\}.
\end{align*}
For $g \in L_{N}^2(T)$, we define the norms for the nonlinear terms as
\begin{align*}
& \|g\|_{G^{s,s'}(T)} = \bigl( \sum_{N,M} M^{2 s'}N^{2s} \bigl\| \eta_{M}(\xi_d)  \| P_N  \F_{x_d}g\|_{G_N(T)}\bigr\|_{L_{\xi_d}^2}^2\bigr)^{\frac12},\\
& \|g\|_{Y^{s,s'}(T)} = \bigl( \sum_{N,M} M^{2 s'}N^{2s} \bigl\| \eta_{M}(\xi_d)  \| P_N  \F_{x_d}g\|_{L_t^1 L_x^{2}}\bigr\|_{L_{\xi_d}^2}^2\bigr)^{\frac12}.
\end{align*}

\subsection*{Linear estimates}
In this subsection we collect the estimates for the flow of the linear Schr{\"o}dinger equation which we employ in the following. Besides the classical Strichartz estimates, we crucially rely on local smoothing and maximal function estimates. The local smoothing estimates follow from (4.18) in~\cite{IK07}. The maximal function estimates for $d \geq 4$ are stated in~(4.6) in~\cite{IK07}. We refer to~\cite{IK06} (see (3.28) in the proof of Lemma~3.3) for the maximal function estimates for $d=3$. See also Lemma~3.2 in~\cite{BIKT11}.

\begin{lem}\label{lemma-2.1}
For all $f \in L^2(\R^{d-1})$, $N
\geq 1$ and $\vece \in \Sp^{d-2}$, we have:\\
\textnormal{(a) (Local smoothing estimate).}
\[
\|e^{it \Delta'}P_{N,\vece}f \|_{L_{\vece}^{\infty,2}} \lesssim N^{-\frac12}\|f \|_{L^2},
\quad (N>1).
\]
\textnormal{(b) (Maximal function estimate).}
\[
\|e^{it \Delta'}P_{N}f \|_{L_{\vece}^{2,\infty}} \lesssim N^{\frac{d-2}{2}}\|f \|_{L^2}, \quad d\geq4,
\]
and
\[
\|\eta(t)e^{it \Delta'}P_{N}f \|_{L_{\vece}^{2,\infty}} \lesssim (1+\log N) N^{\frac{1}{2}}\|f \|_{L^2}, \quad d=3.
\]
\textnormal{(c) (Strichartz estimate).}
\[
\|e^{it \Delta'}f \|_{L_{t,x}^{p_{d-1}}} \lesssim \|f \|_{L^2}.
\]
\end{lem}

Note that Lemma \ref{lemma-2.1} implies
\begin{equation} \label{est:LinSchrInF}
	\|e^{it \Delta'-t \partial_{x_d}} \varphi\|_{F^{s,s'}(T)} \lesssim \|\varphi\|_{H^{s,s'}(\R^d)}
\end{equation}
for all $\varphi \in H^{s,s'}(\R^{d})$.

In order to prove the local well-posedness theory via a fixed point argument, we also need estimates for the inhomogeneous terms in our function spaces. In the case $d \geq 4$, these are provided by Proposition~3.8 in~\cite{BIKT11}, as one sees by checking the definitions of the involved norms. We provide a proof here to include the case $d = 3$.

\begin{lem}\label{lemma-2.2}
We have
\[
\Bigl\| \int_0^t e^{i(t-s) \Delta'} (u(s)) ds \Bigr\|_{F_N(T)} \lesssim \|u\|_{G_N(T)}
\]
for all $0<T<1$.
\end{lem}

\begin{proof}
It is straightforward to prove Lemma \ref{lemma-2.2} in the case $N=1$. Thus, we assume $N>1$. We need to show
\begin{align}
& \Bigl\| \int_0^t e^{i(t-s) \Delta'} (u(s)) ds \Bigr\|_{F_N(T)} \lesssim \|u \|_{L_{t,x}^{p_{d-1}'}},\label{est:lemma2.2-1}\\
& \Bigl\| \int_0^t e^{i(t-s) \Delta'} (u(s)) ds \Bigr\|_{F_N(T)} \lesssim N^{-\frac12} \sum_{j = 1}^{d-1} \|u\|_{L_{\vece_j}^{1,2}}.\label{est:lemma2.2-2}
\end{align}

The former estimate \eqref{est:lemma2.2-1} is a consequence of the Christ-Kiselev lemma. See~\cite{CK01} and~Lemma B.3 in~\cite{WHH09}.
Alternatively, $U^p$ and $V^p$ spaces were employed to show \eqref{est:lemma2.2-1}, see the proof of Lemma~7.3 in~\cite{BIKT11}.
For the latter estimate \eqref{est:lemma2.2-2}, we follow the proof of Lemma~7.4 in~\cite{BIKT11}, see also \cite{P18}.
Because of~\eqref{eq:DecompositionIdentity}, we have
	\begin{equation}\label{eq:dec-pe}
		P_N f = \sum_{j=1}^{d-1} P_{N,\vece_j} \Big[\prod_{l=1}^{j-1} (1-P_{N,\vece_l})\Big]P_N f.
	\end{equation}
Hence, without loss of generality, it suffices to show
\begin{equation}
\Bigl\| P_N P_{N,\vece_1}\int_0^t e^{i(t-s) \Delta'} (u(s)) ds \Bigr\|_{F_N(T)} \lesssim N^{-\frac12} \|u\|_{L_{\vece_1}^{1,2}}.
\end{equation}
We define the fundamental solution of the Schr\"{o}dinger equation in $\R^{d-1}$ as
\[
K_0(t,x)=(4 \pi it)^{-\frac{d-1}{2}} e^{\frac{i|x|^2}{4t}}.
\]
Then, the inhomogeneous term can be expressed as
\begin{align*}
& \int_0^t e^{i(t-s) \Delta'} (u(s)) ds\\
 & = \int_{s<t} \int_{\R^{d-1}}K_0(t-s,x-y)u(s,y) dy ds - \int_{s<0} e^{i(t-s) \Delta'} (u(s)) ds \\
& = \int_{\R}\int_{s<t} \int_{\R^{d-2}} K_0(t-s,x_1-y_1, x'-y')u(s,y_1,y') dy' ds d y_1- e^{it \Delta'}F(x)\\
& = \int_{\R} v_{y_1}(t,x) d y_1- e^{it \Delta'}F(x),
\end{align*}
where
\begin{align*}
& v_{y_1}(t,x)  =\int_{s<t} \int_{\R^{d-2}} K_0(t-s,x_1-y_1, x'-y')u(s,y_1,y') dy' ds,\\
& F(x) = \int_{\R} e^{-is \Delta'} \bigl( \mathbf{1}_{(-\infty, 0)}u(s) \bigr) ds.
\end{align*}
Since the latter term can be handled by Lemma~2.1 and the dual estimate of the local smoothing estimate, it suffices to prove
\begin{equation}
\|P_N P_{N,\vece_1}v_{y_1}\|_{F_N(T)} \lesssim N^{-\frac12} \|u(y_1)\|_{L^2}.
\end{equation}
To see this, we invoke Lemma~7.5 in~\cite{BIKT11} which implies that there exist functions $v_0$ and $w$ such that
\begin{align*}
P_N P_{N,\vece_1}& v_{y_1}(t,x) = (P_{<2^{-40}N, \vece_1} \mathbf{1}_{\{x_1>y_1\}}) \cdot P_N P_{N,\vece_1} e^{it \Delta'}v_0 + w(t,x),\\
& \|v_0\|_{L^2} + N^{-1}(\|\Delta' w\|_{L^2} + \| \partial_t w\|_{L^2}) \lesssim N^{-\frac12} \|u(y_1)\|_{L^2},
\end{align*}
where $P_{<2^{-40}N, \vece_1}$ is defined in the obvious way.
The Sobolev embeddings in time and space yield the necessary bound for $w(t,x)$. To bound the first summand in the previous decomposition, we write as in~\cite{BIKT11} for any $1 \leq j \leq d-1$ 
\[
P_{N,\vece_j}[(P_{<2^{-40}N, \vece_1} \mathbf{1}_{\{x_1>y_1\}}) \cdot P_N  v] = \sum_{N_1 \sim N}
P_{N,\vece_j}[(P_{<2^{-40}N, \vece_1} \mathbf{1}_{\{x_1>y_1\}}) \cdot P_N  P_{N_1,\vece_j}v].
\]
Consequently, the desired bound for the first summand follows from the linear estimates in Lemma~2.1.
\end{proof}
As above, we point out that Lemma \ref{lemma-2.2} yields
\begin{equation} \label{est:InhomSchrInF}
\Bigl\| \int_0^t e^{(t-s)(i \Delta' - \partial_{x_d})} (u(s)) ds \Bigr\|_{F^{s,s'}(T)} \lesssim \|u\|_{G^{s,s'}(T)}
\end{equation}
for all $0<T<1$ and $u \in G^{s,s'}(T)$.

\section{Local well-posedness of the degenerate Zakharov system}\label{sec:lwp}

For the purpose of this paper it is more convenient to work with the first order reformulation of the degenerate Zakharov system~\eqref{eq:dZakharov}. Setting $N = n - i |\nabla'|^{-1} \partial_t n$, system~\eqref{eq:dZakharov} is equivalent to
\begin{equation}\label{eq:dZakharovfo}
  \begin{split}
  i(\partial_t E + \partial_{x_d} E) + \Delta' E={}&\Re(N) E,\quad \text{in } (-T,T) \times \R^d,\\
  i \partial_t N + |\nabla'| N ={}& - |\nabla'| |E|^2,  \quad \text{in } (-T,T) \times \R^d.
\end{split}
\end{equation}
The initial condition \eqref{eq:ic} transforms into
\begin{equation}\label{eq:icfo}
(E,N)|_{t=0} =(E_0,N_0),
\end{equation}
where $N_0 = n_0 - i |\nabla'|^{-1} n_1$. Note that $N_0$ belongs to $H^{s-\frac{1}{2},s'}(\R^d)$ if and only if $(n_0,|\nabla'|^{-1} n_1) \in \big( H^{s-\frac{1}{2},s'}(\R^d) \big)^2$. Moreover, the term $\overline{N}$ can be treated in the same way as $N$ in our analysis so that we drop the real part in~\eqref{eq:dZakharovfo} for simplicity. We remark that one can use a modified transformation involving $(1-\Delta')^{-\frac12}$ to avoid any low frequency conditions by following the argument in \cite{BHHT09}, we omit the details.

Besides the estimates from Section~\ref{sec:pre}, the crucial ingredients in the proof of the local well-posedness theorem are the following estimates for the nonlinear terms appearing on the right-hand side of~\eqref{eq:dZakharovfo}.

\begin{prop} \label{prop:NonlinearEstimates}
Let $d \geq 3$, $s > \frac{d-2}{2}$, and $s'>\frac12$. Then we have
\begin{align}
& \|u v \|_{G^{s,s'}(T)} \lesssim T^{\frac12}\|u\|_{F^{s,s'}(T)} \|v\|_{W^{s-\frac12,s'}(T)},\label{est:prop2.3-a}\\
& \| |\nabla'| (u_1 u_2)\|_{Y^{s-\frac12,s'}(T)} \lesssim T^{\frac12} \|u_1\|_{F^{s,s'}(T)} \|u_2\|_{F^{s,s'}(T)}\label{est:prop2.3-b}
\end{align}
for all $T > 0$.
\end{prop}

We postpone the proof of Proposition~\ref{prop:NonlinearEstimates} to Section~\ref{sec:NonlinearEstimates} and first show how it implies the local well-posedness of the degenerate Zakharov system.

\begin{thm}
	\label{thm:lwp}
	Let $d \geq 3$, $s > \frac{d-2}{2}$, and $s' > \frac{1}{2}$.
	Then, for every initial data $(E_0, N_0) \in H^{s,s'}(\R^d) \times H^{s-\frac{1}{2},s'}(\R^d)$ there is a time $T > 0$ and a unique solution
	\begin{align*}
		(E,N) \in C([-T,T],H^{s,s'}(\R^d) \times H^{s-\frac{1}{2},s'}(\R^d)) \cap (F^{s,s'}(T) \times W^{s-\frac{1}{2},s'}(T))
	\end{align*}
	of the degenerate Zakharov system~\eqref{eq:dZakharovfo}--\eqref{eq:icfo}.
\end{thm}

\begin{proof}
	We define a mapping $\Phi$ by the right-hand side of the integral equation corresponding to~\eqref{eq:dZakharovfo} (after dropping the real part), i.e.
	\begin{align*}
		\Phi(E,N) = \begin{pmatrix}
						e^{i t \Delta' - t \partial_{x_d}} E_0 - i \int_0^t e^{(t-s)(i \Delta' - \partial_{x_d})} (N E)(s)  d s \vspace{0.2em}\\
						e^{i t |\nabla'|} N_0 + i \int_0^t e^{i(t-s)|\nabla'|} \, |\nabla'| |E(s)|^2 d s
					\end{pmatrix}.
	\end{align*}
	The estimates~\eqref{est:LinSchrInF} and \eqref{est:InhomSchrInF}, the energy estimate for the half-wave equation, and the nonlinear estimates from Proposition~\ref{prop:NonlinearEstimates} now allow us to perform a standard fixed point argument in the Banach space
	\begin{align*}
		C([-T,T],H^{s,s'}(\R^d) \times H^{s-\frac{1}{2},s'}(\R^d)) \cap (F^{s,s'}(T) \times W^{s-\frac{1}{2},s'}(T)),
	\end{align*}
	which yields the assertion of the theorem.
\end{proof}


\section{Nonlinear Estimates}
\label{sec:NonlinearEstimates}

We now provide the proof of the nonlinear estimates in Proposition~\ref{prop:NonlinearEstimates}.

\vspace{1em}

\emph{Proof of Proposition~\ref{prop:NonlinearEstimates}.}
Here we only consider $d \geq 4$. The case $d=3$ can be handled in a similar way.
We consider \eqref{est:prop2.3-a} first. By Minkowski's inequality, we get that
\begin{align*}
&\| u v \|_{G^{s,s'}(T)}
\lesssim \bigl( \sum_{N,M} M^{2 s'}N^{2s} \bigl\| \eta_{M}(\xi_d)  \| P_N \bigl( \F_{x_d}u *_{\xi_d} \F_{x_d} v \bigr)\|_{G_N(T)}\bigr\|_{L_{\xi_d}^2}^2\bigr)^{\frac12}\\
& \lesssim
\bigl( \sum_{N,M} M^{2 s'}N^{2s} \bigl\| \eta_{M}(\xi_d) \int_{\R}  \| P_N \bigl( (\F_{x_d}u)(\xi_d-\eta_d) (\F_{x_d} v)(\eta_d) \bigr)\|_{G_N(T)} d\eta_d \bigr\|_{L_{\xi_d}^2}^2\bigr)^{\frac12},
\end{align*}
where $*_{\xi_d}$ denotes the convolution in the variable $\xi_d$.
Hence, it suffices to show the estimate
\begin{equation}\label{est:prop2.3-a1}
\| P_N (f_1 f_2)\|_{G_N(T)} \lesssim T^{\frac12}N_2^{-\frac12} N_{\min}^{\frac{d-2}{2}} \| f_1\|_{F_{N_1}(T)} \|f_2\|_{L_t^{\infty} L_x^2}
\end{equation}
for all $f_{1} \in F_{N_1}(T)$ and $f_{2} \in L_{N_2}^2(T)\cap L^\infty_t L^2_x$, where $N_{\min}=\min(N_1,N_2)$.

We consider the three cases $N \ll N_1 \sim N_2$, $N_2 \ll N \sim N_1$, and $N_1 \lesssim N \sim N_2$. In the first case, using Bernstein's and H\"{o}lder's inequality, we find that
\begin{align}
\label{eq:Estf1f2InGNI}
\| P_N (f_1 f_2)\|_{G_N(T)} & \leq \| P_N(f_1f_2)\|_{L_{t,x}^{p_{d-1}'}} \lesssim T^{\frac12}N^{\frac{d-3}{2}} \| P_N(f_1f_2)\|_{L_t^{\alpha}L_x^{\beta}} \\
& \lesssim T^{\frac12}N^{\frac{d-3}{2}} \|f_1\|_{L_t^{\alpha}L_x^{\gamma}} \|f_2\|_{L_t^{\infty}L_x^2}
\lesssim T^{\frac12}N^{\frac{d-3}{2}} \| f_1\|_{F_{N_1}(T)} \|f_2\|_{L_t^{\infty} L_x^2}, \nonumber
\end{align}
where
\[
\Bigl( \frac{1}{\alpha}, \, \frac{1}{\beta},\, \frac{1}{\gamma}\Bigr) = \Bigl(\frac{1}{d+1}, \, 1-\frac{2}{(d-1)(d+1)}, \, \frac12-\frac{2}{(d-1)(d+1)}\Bigr).
\]
In the last estimate we also used that $\|f_1\|_{L_t^{\alpha}L_x^{\gamma}} \lesssim \|f_1\|_{F_{N_1}(T)}$ which follows by interpolating between $L_t^{\infty} L_x^2$ and $L_{t,x}^{p_{d-1}}$.
The second case $N_2 \ll N \sim N_1$ can be treated in a similar way, employing H{\"o}lder's inequality first and then Bernstein's inequality on $f_2$.

It remains the case $N_1 \lesssim N \sim N_2$. If $N = 1$, we again argue as in~\eqref{eq:Estf1f2InGNI}. If $N > 1$, H\"{o}lder's inequality yields
\begin{align*}
\| P_N (f_1 f_2)\|_{G_N(T)} & \leq N^{-\frac12} \sum_{j=1}^{d-1} \|P_N(f_1f_2)\|_{L_{\vece_j}^{1,2}} \lesssim N^{-\frac12} \sum_{j=1}^{d-1} \|f_1\|_{L_{\vece_j}^{2,\infty}} \|f_2\|_{L_{t,x}^2} \\
&\lesssim T^{\frac12}N^{-\frac12}  N_1^{\frac{d-2}{2}} \| f_1\|_{F_{N_1}(T)} \|f_2\|_{L_t^{\infty} L_x^2},
\end{align*}
which completes the proof of \eqref{est:prop2.3-a1}.

Next we prove \eqref{est:prop2.3-b}. We compute that
\begin{align*}
&\| |\nabla'|(u_1 u_2)\|_{Y^{s-\frac12, s'}} \\
&\lesssim \! \bigl( \sum_{N,M} M^{2 s'}N^{2s+1} \bigl\| \eta_{M}(\xi_d) \! \int_{\R}  \| P_N \bigl( (\F_{x_d}u_1)(\xi_d-\eta_d) (\F_{x_d} u_2)(\eta_d) \bigr)\|_{L_{t}^1L_x^2} d\eta_d \bigr\|_{L_{\xi_d}^2}^2\bigr)^{\frac12}.
\end{align*}
Therefore, it is enough to show
\begin{equation}\label{est:prop2.3-b1}
\| P_N (g_1 g_2)\|_{L_{t,x}^2} \lesssim N_{\max}^{-\frac12} N_{\min}^{\frac{d-2}{2}} \| g_1\|_{F_{N_1}(T)} \|g_2\|_{F_{N_2}(T)}
\end{equation}
for all $g_1 \in F_{N_1}(T)$ and $g_2 \in F_{N_2}(T)$,
where $N_{\max}=\max(N_1,N_2)$ and $N_{\min}=\min(N_1,N_2)$. Without loss of generality, we may assume $N_2 \leq N_1$. In the case $N_1 = 1$, we easily obtain~\eqref{est:prop2.3-b1} from H{\"o}lder's inequality, Bernstein's inequality and interpolation.

 We can thus assume $N_1 > 1$ in the following. Recall from \eqref{eq:dec-pe} that we have
	\begin{align*}
		g_1 = \sum_{j=1}^{d-1} P_{N_1,\vece_j} \Big[\prod_{l=1}^{j-1} (1-P_{N_1,\vece_l})\Big] g_1.
	\end{align*}
Since $(P_{N_1/2} + P_{N_1} + P_{2N_1})P_{N_1,\vece}$ is bounded on $L^{\infty,2}_{\vece'}$ for all $\vece, \vece' \in \Sp^{d-2}$, H\"{o}lder's inequality allows us to estimate
\begin{align*}
\| P_N ( g_1 g_2)\|_{L_{t,x}^2} & \leq \sum_{j = 1}^{d-1} \Big\|\Big(P_{N_1,\vece_j} \Big[\prod_{l=1}^{j-1} (1-P_{N_1,\vece_l})\Big] g_1 \Big) g_2\Big\|_{L^2_{t,x}} \\
&\lesssim \sum_{j = 1}^{d-1} \|P_{N_1,\vece_j} g_1\|_{L_{\vece_j}^{\infty,2}} \|g_2\|_{L_{\vece_j}^{2,\infty}} \lesssim N_1^{-\frac12} N_2^{\frac{d-2}{2}} \| g_1\|_{F_{N_1}(T)} \|g_2\|_{F_{N_2}(T)},
\end{align*}
which completes the proof of \eqref{est:prop2.3-b1}. \hfill $\qed$

\vspace{1em}

The above proof shows that, if $d\geq 4$, in the case $s=\frac{d-2}{2}$ and $s'=\frac12$ one obtains similar estimates in the $\ell^1$-based Besov norms.

\appendix

\section{Examples involving Fourier restriction norms}\label{sec:ex}
In this section, we prove that it is impossible to solve the problem by using Fourier restriction norms only. To that end, we define the additional frequency and modulation projections $P_{N,M}$, $\mathcal{E}_{L}$, and
$\mathcal{W}_{L}^{\pm}$ as
\begin{align*}
&(\F_{x,x_d} P_{N,M}f )(\xi,\xi_d ):=  \eta_{N}(\xi ) \eta_{M}(\xi_d)(\F_{x,x_d}{f})(\xi,\xi_d ),\\
 &\widehat{\mathcal{E}_{L} u}(\tau ,\xi,\xi_d ):=  \eta_{L}(\tau + |\xi|^2 + \xi_d)\widehat{u}(\tau ,\xi,\xi_d ), \;
\widehat{\mathcal{W}_{L}^{\pm} v}(\tau ,\xi,\xi_d ):=  \eta_{L}(\tau \pm |\xi|)\widehat{v}(\tau ,\xi,\xi_d ).
\end{align*}
For the parameters $s,s',b \in \R$ and $1\leq p\leq \infty$ we define the Fourier restriction spaces $X^{s,s',b,p}_{\mathcal{E}} (\R^{d+1})$ and $X^{s,s',b,p}_{\mathcal{W}_{\pm}} (\R^{d+1})$ as the collection of tempered distributions such that the following norms are finite:
\begin{align*}
  & \|u  \|_{X^{s,s',b,p}_{\mathcal{E}}} := \Bigl\|\Bigl( N^{s} {M}^{ s'}  \bigl\|\bigl( L^{b } \|  P_{N,M} \mathcal{E}_L u \|_{L_{t ,x ,x_d}^{2}}\bigr)_{L \in 2^\N}\bigr\|_{\ell^p_L}\Bigr)_{N,M \in 2^\N}\Bigr\|_{\ell^2_{N,M}},\\
  & \|u  \|_{X^{s,s',b,p}_{\mathcal{W}_{\pm}}} := \Bigl\|\Bigl( N^{s} {M}^{ s'}  \bigl\|\bigl( L^{b } \|  P_{N,M} \mathcal{W}^{\pm}_L u \|_{L_{t ,x ,x_d}^{2}}\bigr)_{L \in 2^\N}\bigr\|_{\ell^p_L}\Bigr)_{N,M \in 2^\N}\Bigr\|_{\ell^2_{N,M}}.
\end{align*}
The precise statement we prove in this section is the following:
\begin{prop}\label{thm:ex}
\begin{enumerate}
\item[]
\item \label{it:b1} Suppose that there exists $C>0$ such that
\[
\| u \, v \|_{X^{s,s',b_1-1,p_1}_{\mathcal{E}}} \leq C \| u\|_{X^{s,s',b_1,p_1}_{\mathcal{E}}} \|v  \|_{X^{s-\frac12, s',b_2,p_2}_{\mathcal{W}_{\pm}}}
\]
holds for all square-integrable $u,v$ with compact Fourier support.
Then, either $b_1<1/2$ or $(b_1,p_1)=(1/2,\infty)$ holds.
\item \label{it:b2} Suppose that there exists $C>0$ such that
\[
\| \sqrt{-\Delta'}(|w|^2) \|_{X^{s-\frac12, s',b_2-1,p_2}_{\mathcal{W}_{\pm}}} \leq C \| w\|_{X^{s,s',b_1,p_1}_{\mathcal{E}}}^2
\]
holds for all square-integrable $w$ with compact Fourier support.
Then, either $b_1>1/2$ or $(b_1,p_1)=(1/2,1)$ holds.
\end{enumerate}
\end{prop}
\begin{proof}
For $r>0$ and $a \in \R^{d+1}$, we define the ball $B_r(a) = \{x \in \R^{d+1} \, | \, |x-a| \leq r \}$.
Let $N \gg 1$, $L\geq 1$, $a_N^{\pm}=(\mp N, N, 0, \ldots, 0, -N^2 \pm N)$.
We use $S_L=\{(\tau,\xi,\xi_d) \in \R^{d+1} \, | \, L \leq |\tau+|\xi|^2 +\xi_d| \leq 2L\}$.

Firstly, we show \eqref{it:b1}. We define the functions $u$, $v_{N,\pm} \in L^2(\R^{d+1})$ as
\[
\widehat{u} = \chi_{B_1(0)}, \quad \widehat{v}_{N,\pm}= N^{-s-2 s' + 1/2} \chi_{B_1(a_N^{\pm})},
\]
where $\chi_A$ denotes the characteristic function of the set $A$.
It is easily seen that for all $b_1$, $b_2 \in \R$, $p_1$, $p_2 \in [1,\infty]$, it holds that
\[
\| u\|_{X^{s,s',b_1,p_1}_{\mathcal{E}}} \sim 1, \quad \|v_{N,\pm}  \|_{X^{s-\frac12, s',b_2,p_2}_{\mathcal{W}_{\pm}}} \sim 1.
\]
Thus, it suffices to show that if $p_1 < \infty$, we have
\begin{equation}\label{est:theorem2.1-1}
\lim_{N \to \infty} \| u \, v_{N,\pm} \|_{X^{s,s',-\frac12,p_1}_{\mathcal{E}}} = \infty.
\end{equation}
We observe that if $(\tau,\xi,\xi_d) \in B_{1/2}(a_N^{\pm})$ then $(\chi_{B_1(0)} * \chi_{B_1(a_N^{\pm})}) (\tau,\xi,\xi_d) \sim 1$. Therefore, for $p_1 < \infty$, we get
\begin{align*}
\| u \, v_{N,\pm} \|_{X^{s,s',- \frac12,p_1}_{\mathcal{E}}} & \geq N^{s+2 s'}  \bigl( \sum_{1 \leq L \leq 2^{-2}N} L^{-\frac{p_1}{2}} \| \widehat{u} * \widehat{v}_{N,\pm} \|_{L_{\tau, \xi, \xi_d}^{2}(B_{1/2}(a_N^{\pm}) \cap S_L)}^{p_1} \bigr)^{\frac{1}{p_1}}\\
& \geq N^{\frac12}\bigl( \sum_{1 \leq L \leq 2^{-2} N} L^{-\frac{p_1}{2}} \| \chi_{B_1(0)} * \chi_{B_1(a_N^{\pm})}  \|_{L_{\tau, \xi, \xi_d}^{2}(B_{1/2}(a_N^{\pm}) \cap S_L)}^{p_1} \bigr)^{\frac{1}{p_1}}\\
& \sim  N^{\frac12}\bigl( \sum_{1 \leq L \leq 2^{-2}N} L^{-\frac{p_1}{2}} \Bigl( \frac{L}{N}\Bigr)^{\frac{p_1}{2}} \bigr)^{\frac{1}{p_1}} \sim (\log N)^{\frac{1}{p_1}},
\end{align*}
which implies \eqref{est:theorem2.1-1}.
Here, the third estimate holds because the measure of the set $B_{1/2}(a_N^{\pm}) \cap S_L$ is comparable to $L/N$ if $1 \leq L \leq 2^{-2}N$.

Secondly, we prove \eqref{it:b2}. Let $1<p_1 < \infty$. We define the functions $w_{N,\pm} \in L^2(\R^{d+1})$ as
\[
\widehat{w}_{N,\pm} = \chi_{B_1(0)} +
(\log N)^{-\frac{1}{p_1}}N^{-s-2 s' + \frac12} \sum_{1 \leq L \leq 2^{-2}N} L^{-1}\chi_{B_{1}(a_N^{\pm}) \cap S_L}.
\]
It is straightforward to check $\| w_{N,\pm}\|_{X^{s,s',\frac12,p_1}_{\mathcal{E}}} \sim 1$.
Our goal is to show that for all $b_2$, $p_2$ it holds that
\begin{equation}\label{est:theorem2.1-2}
\lim_{N \to \infty}\| \sqrt{-\Delta'}(|w_{N,\pm}|^2) \|_{X^{s-\frac12, s',b_2-1,p_2}_{\mathcal{W}_{\pm}}} = \infty.
\end{equation}
To see this, we note that if $(\tau,\xi,\xi_d) \in B_{1/2}(a_N^{\pm})$, then \[(\chi_{B_1(0)} * \chi_{B_1(a_N^{\pm}) \cap S_L}) (\tau,\xi,\xi_d) \sim L/N\] holds.
We compute that
\begin{align*}
&\| \sqrt{-\Delta'}(|w_{N,\pm}|^2) \|_{X^{s-\frac12, s',b_2-1,p_2}_{\mathcal{W}_{\pm}}}\\ &\geq{} 
(\log N)^{-\frac{1}{p_1}}N
\Bigl\| \sum_{1 \leq L \leq 2^{-2}N} L^{-1} \bigl( \chi_{B_1(0)}*  \chi_{B_{1}(a_N^{\pm}) \cap S_L} \bigr) \Bigr\|_{L_{\tau, \xi, \xi_d}^2(B_{1/2}(a_N^{\pm}))}\\
&\sim{}  (\log N)^{-\frac{1}{p_1}}N \Bigl\| \sum_{1 \leq L \leq 2^{-2}N} L^{-1} \Bigl( \frac{L}{N}\Bigr) \Bigr\|_{L_{\tau, \xi, \xi_d}^2(B_{1/2}(a_N^{\pm}))} \sim (\log N)^{1-\frac{1}{p_1}}.
\end{align*}
This completes the proof of \eqref{est:theorem2.1-2}.
\end{proof}

\section*{Acknowledgments}
Financial support by the German Research Foundation (DFG) through the CRC 1283 ``Taming uncertainty and profiting from
  randomness and low regularity in analysis, stochastics and their
  applications'' is acknowledged. The second author is supported by JSPS KAKENHI Grant Number 820200500051.
\bibliographystyle{abbrv}
\bibliography{DegZakhBiblio}

\end{document}